\documentclass[a4paper, 10pt, conference]{ieeeconf}      

\IEEEoverridecommandlockouts                              
\usepackage{amsmath} 
\usepackage{amssymb}  

\newcommand\numberthis{\addtocounter{equation}{1}\tag{\theequation}}

\newcommand{\8}{\infty}
\newcommand{\R}{\mathbb{R}}
\newcommand{\N}{\mathbb{N}}
\newcommand{\s}{\mathbb{S}}

\newtheorem{mylem}{Lemma}

\newtheorem{mythm}{Theorem}

\title{\LARGE \bf
A Convex Approach for Stability Analysis of Coupled PDEs with Spatially Dependent Coefficients
}

\author{Evgeny Meyer and Matthew M. Peet
\thanks{This work was supported by the National Science Foundation under Grants No. 1301851 and 1301660}
\thanks{E. Meyer is a Ph.D student with the School for Engineering of
Matter, Transport and Energy, Arizona State University, Tempe, AZ, 85281
USA
        {\tt\small edmeyer@asu.edu}}%
\thanks{ M. M. Peet is an assistant professor with the School for Engineering
of Matter, Transport and Energy, Arizona State University, Tempe, AZ,
85281 USA
        {\tt\small mpeet@asu.edu}}%
}

\begin{document}

\maketitle
\thispagestyle{empty}
\pagestyle{empty}

\begin{abstract}

In this paper, we present a methodology for stability analysis of a general class of systems defined by coupled Partial Differential Equations (PDEs) with spatially dependent coefficients and a general class of boundary conditions. This class includes PDEs of the parabolic, elliptic and hyperbolic type as well as coupled systems without boundary feedback. Our approach uses positive matrices to parameterize a new class of SOS Lyapunov functionals and combines these with a parametrization of projection operators which allow us to enforce positivity and negativity on subspaces of $L_2$. The result allows us to express Lyapunov stability conditions as a set of Linear Matrix Inequality (LMI) constraints which can be constructed using SOSTOOLS and tested using Semi-Definite Programming (SDP) solvers such as SeDuMi or Mosek. The methodology is tested using several simple numerical examples and compared with results obtained from simulation using a standard form of numerical discretization.
\end{abstract}

\section{INTRODUCTION}

Partial Differential Equations (PDEs) are often used to model systems in which the quantity of interest varies continuously in both space and time.
Examples of such quantities include: deflection of beams (Euler-Bernoulli equation); velocity and pressure of fluid flow (Navier-Stokes equations); and population density in predator-prey models.
See \cite{evans1998}, \cite{garabedian1964partial} and \cite{john1982partial} for a wide range of examples.

In this paper we address the stability analysis of a general class of coupled linear PDEs with spatially dependent coefficients, i.e.
\[
u_t(t,x)=A(x)u_{xx}(t,x)+B(x)u_x(t,x)+C(x)u(t,x),\numberthis{\label{class}}
\]
where $u:[0,\8)\times[a,b]\to\R^n$ and $A,B,C$ are polynomial matrices. Boundary conditions have the general form $D\cdot [u(a), u(b), u_x(a),  u_x(b)]^T=0$ where $D \in \R^{m\times 4n}$. PDEs expressed in this form can be of the parabolic, elliptic, or hyperbolic type. As can be seen in Section 3 such PDEs as, for example, Schrodinger and accoustic wave equations can be expressed in the form of Equation~\eqref{class}.

Recently, there has been some progress in theory of analyzing and controlling PDE models of this form. First, we note the development of a theory of state-space for systems of PDEs or DDEs called Semigroup Theory, wherein the state of the system belongs to a certain space of functions. The solution map for these systems is an operator-valued function (``strongly continuous semigroup'' - SCS), indexed to the time domain, which maps the current state to a future state. For an introduction to Semigroup Theory we refer readers to \cite{lasiecka1980unified}, \cite{curtain1995introduction}.

In the semigroup framework, stability, controllability and observability conditions can be expressed using operator inequalities in the same way that LMIs are used to represent those properties for ODEs. As an example, for a system $\dot u=\mathcal{A}u$ which defines a SCS on a Hilbert space $X$ with $\mathcal{A}$ being the infinitesimal generator, the exponential stability of the system is equivalent to the existence of a positive bounded linear operator $\mathcal{P}:X\to X$ such that
\[
\left<u,\mathcal{AP}u\right>_X+\left<\mathcal{A}u,\mathcal{P}u\right>_X\le - \left<u,u\right>_X\numberthis{\label{LOI}}
\]
for all $u$ in the domain of $\mathcal{A}$. Condition (\ref{LOI}) is termed a Linear Operator Inequality (LOI).
The terminology LOI is deliberately chosen to suggest a parallel to the use of Linear Matrix Inequalities (LMIs) for computational analysis and control of ODEs.
Indeed, there have been efforts to use discretization to solve LOI type conditions for stability analysis and optimal control of PDEs (see, e.g.~\cite{christofides2012nonlinear}), optimal actuator placement for parabolic PDEs (see \cite{demetriou2003optimization} and \cite{morris2015using}).
However, in this paper, we do not employ discretization. While discretization has proven quite effective in practice, one should note that in general it is difficult to determine if feasibility of the discretized LOI implies stability of the non-discretized PDE. In contrast, this paper is focused on exploring how to use computation to solve LOIs (\ref{LOI}) directly by parameterizing the cone of positive and negative operators.

An alternative approach to control (but not stability analysis) of PDEs is backstepping - See \cite{krstic2008}, \cite{vazquez2016boundary}.
The backstepping approach searches for a mapping from the original PDE to a chosen stable PDE using a Volterra operator. The desired controller is then found by formulating a PDE for the kernel of the Volterra operator - the solution to which yields a stabilizing controller for the original system. An alternative approach, taken by \cite{peet2014lmi}, uses some of the machinery developed for DDEs to express Lyapunov inequalities as LMIs, which can then be tested using standard interior-point algorithms.
We also note that in \cite{fridman2016new} stability analysis and initial state recovery of semi-linear wave equation are presented in terms of LMIs.

Recently, our lab, in collaboration with other researchers have begun to explore how to use the SOS method for optimization of polynomials to study analyze and control PDEs without the need for discretization. Specifically, in~\cite{papachristodoulou2006}, we considered stability analysis of scalar nonlinear PDEs using a simple form of Lyapunov function. This simple Lyapunov function was recently extended in~\cite{valmorbida2014semi} and in~\cite{valmorbida2015stability} to consider some forms of coupled PDEs and in~\cite{Ahmadi2016} to consider passivity. In~\cite{gahlawat2015convex} and related publications, the class of Lyapunov functions was expanded to squares of semi-separable integral operators and applied to output-feedback dynamic control of scalar PDEs. Finally, in~\cite{meyer2015stability}, we considered stability of PDEs with multiple spatial variables.

The goal of this paper is, for the first time, to use SOS Lyapunov functionals defined by combined multiplier and integral operators to study stability of systems of coupled PDEs. Specifically, we parameterize Lyapunov functionals of the following form
\begin{align*}
V(w)=&\int_a^b w(x)^T M(x)w(x)\,dx \\
&+ \int_a^b w(x)^T\int_a^b N(x,y)w(y)\,dydx\numberthis{\label{functionals}}
\end{align*}
where $w\in L_2^n(a,b)$, and $M,N$ are polynomial matrices.
Using Lyapunov functionals of this form, the problem of stability of \textit{coupled} PDEs, in particular, is difficult in that the coefficients of the PDE are matrix-valued and hence do not commute with the polynomial matrices which define the Lyapunov functions. This issue makes it difficult to manipulate the derivative of the Lyapunov functional into a form for which we can test negativity on $L_2$. This is complicated by the presence of spatial derivatives in the time-derivative of the functional. To address this problem, we use a generalization of the concept of ``spacing operators'' which allows the algorithm to search over the space of integral equalities defined by the Fundamental Theorem of Calculus and our general form of boundary conditions - a concept initially proposed for scalar PDEs in~\cite{valmorbida2014semi} and~\cite{meyer2015stability}. Numerical results indicate that the proposed algorithm is effective at estimating the stability regions of several classes of coupled PDEs.


\section{NOTATION}
$\R$ and $\N$ denote the sets of real and natural numbers. $I_n$ is the identity matrix of dimemsion $n\times n$. $L_2^n(a,b)$ is the Hilbert space of Lebesgue square  integrable real vector valued functions on the interval $(a,b)\subset\R$, i.e.
\[
L_2^n(a,b):=\left\{f:(a,b)\to\R^n\, \Bigg|\, \sqrt{\int_a^b f(x)^Tf(x)\,dx}<\8\right\}
\]

For a function $u:[0,\8)\times(a,b)\to\R^n$, the classical notation $u_t(t,x)$ and $\ u_{x}(t,x)$ represent partial derivatives with respect to the first and second independent variables.
We also will use the classical notation for the derivative of a function with one variable, i.e. $w',w''$.

\section{CLASS OF SYSTEMS WE CONSIDER}

In this paper we propose an algorithm for stability analysis of the following class of PDEs where the function $u:[0,\8)\times[a,b]\to\R^n$ satisfies
\[
u_t(t,x)=A(x)u_{xx}(t,x)+B(x)u_x(t,x)+C(x)u(t,x)\numberthis{\label{pde:class}}
\]
for all $t>0$ and $x\in(a,b)$. The coefficients $A,B,C$ are polynomial matrices.

We use the matrix $D\in\R^{4n\times 4n}$ to represent boundary conditions, i.e.
for all $t>0$
\[
D\left[ \begin{array}{c}
u(t,a)\\ u(t,b) \\ u_x(t,a) \\ u_x(t,b)
\end{array}  \right]=0.\numberthis{\label{eq:BC}}
\]
Thus, in case of homogeneous Dirichlet boundary conditions
\[
D=\left[ \begin{array}{cccc}
I_n & 0 &0&0 \\
0& I_n & 0&0 \\
0& 0 & 0&0 \\
0& 0 & 0&0
\end{array} \right].\numberthis{\label{Dirichlet}}
\]
Mixed boundary conditions, for example homogeneous Neumann at $x=a$ and Dirichlet at $x=b$, can be written as (\ref{eq:BC}) with
\[
D=\left[ \begin{array}{cccc}
0 & 0 &0&0 \\
0& I_n & 0&0 \\
0& 0 & I_n&0 \\
0& 0 & 0&0
\end{array} \right].\numberthis{\label{Neumann}}
\]
We assume that solution to (\ref{pde:class}) exists, is unique and depends continuously on the initial condition $u(0,x)$.
Also, for each $t>0$ we suppose $u(t,\cdot),u_x(t,\cdot),u_{xx}(t,\cdot)\in L_2^n(a,b)$.

\subsection{Example 1: Schr\"{o}dinger equation}

To illustrate the class of PDEs which can be written as (\ref{class}), we first consider the Schrodinger equation.
In the following equation $V$ is the potential energy, $i$ is the imaginary unit, $\hbar$ is the reduced Planck constant and $\psi$ is the wave function of the quantum system.
\[
i\hbar \psi_t(t,x)=-\frac{\hbar^2}{m}\psi_{xx}(t,x)+V(x)\psi(t,x)
\]
can be written as two coupled PDEs if we decompose the solution into real and imaginary parts as $\psi(t,x)=\psi^{rl}(t,x)+i\psi^{im}(t,x)$ and then separate the real and imaginary parts of the equation to get two coupled PDEs, i.e.
\begin{align*}
\left[\begin{array}{c}
\psi^{rl}_t(t,x)\\
\psi^{im}_t(t,x)
\end{array}\right]&=\underbrace{\frac{\hbar}{m}\left[\begin{array}{cc}
0 & -1\\
1 & 0
\end{array}\right]}_A\left[\begin{array}{c}
\psi^{rl}_{xx}(t,x)\\
\psi^{im}_{xx}(t,x)
\end{array}\right] \\
&\hspace*{0.45in}+\underbrace{\frac{V(x)}{\hbar}\left[\begin{array}{cc}
0 & 1\\
-1 & 0
\end{array}\right]}_{C(x)}\left[\begin{array}{c}
\psi^{rl}(t,x)\\
\psi^{im}(t,x)
\end{array}\right].
\end{align*}
\subsection{Example 2: Model for an Acoustic Wave}
Next consider the following model for a 1-D acoustic wave. This hyperbolic PDE can be written in form (\ref{pde:class}) where
For all $t>0,r\in(0,R)$ and some fixed $c>0$, we define 
\[
p_{tt}(t,r)=c^2p_{rr}(t,r)+\frac{2c^2}{r}p_r(t,r).
\]
This scalar PDE is equivalent to two coupled first order PDEs as
\begin{align*}
\left[\begin{array}{c}
q_t(t,r)\\
p_t(t,r)
\end{array}\right] =& \underbrace{\left[\begin{array}{cc}
0 & c^2 \\
0 & 0
\end{array}\right]}_A \left[\begin{array}{c}
q_{rr}(t,r)\\
p_{rr}(t,r)
\end{array}\right] \\
&\hspace*{0.5in}+ \underbrace{\left[\begin{array}{cc}
0 & \frac{2c^2}{r} \\
0 & 0
\end{array}\right]}_{B(r)} \left[\begin{array}{c}
q_r(t,r)\\
p_r(t,r)
\end{array}\right]\\
&\hspace*{1in}+ \underbrace{\left[\begin{array}{cc}
0 & 0 \\
1 & 0
\end{array}\right]}_C\left[\begin{array}{c}
q(t,r)\\
p(t,r)
\end{array}\right],
\end{align*}
where $q$ is an auxiliary function. If the boundary conditions imply amplification of the waves, i.e.
\[
p(t,0)=f_1p(t,R)\text{ and } p_r(t,0)=f_2p_r(t,R)
\]
for some $f_1,f_2>0$, then in that case the boundary conditions are defined using
\[
D=\left[\begin{array}{cccccccc}
0&1&0&-f_1&0&0&0 &0\\
0&0&0 &0&0&1&0&-f_2\\
0&0&\hdots&\hdots&\hdots&\hdots&\hdots &0\\
\vdots&\ddots&\ddots&\ddots&\ddots&\ddots&\ddots&\vdots\\
0&0&\hdots&\hdots&\hdots&\hdots&\hdots &0
\end{array}\right]
\]


\section{LYAPUNOV STABILITY THEOREM}

\begin{mythm}\label{thLyap}
Let there exist continuous $V:L_2^n(a,b)\to\R$, $l,m\in\mathbb{N}$ and $b,a>0$ such that
\[
a\|w\|_{L_2^n}^l\leq V(w) \leq b \|w\|_{L_2^n}^m
\]
for all $w\in L_2^n(a,b)$. Furthermore, suppose that there exists $c\geq0$ such that for all $t\geq0$ the upper right-hand derivative
\[
D^+[V(u(t,\cdot))]\leq -c\|u(t,\cdot)\|^m_{L_2^n},
\]
where $u$ satisfies (\ref{pde:class}). Then
\[
\|u(t,\cdot)\|_{L_2^n}\leq \sqrt[l]{\frac{b}{a}}\|u(0,\cdot)\|^{m/l}_{L_2^n}\exp\left\{ -\dfrac{c}{lb}t \right\}.
\]
\end{mythm}
\begin{proof}
For the proof see \cite{meyer2015stability}.
\end{proof}


\section{AN SOS PARAMETRIZATION OF POSITIVE FUNCTIONALS}

In this section, we present parameterization of a set of Lyapunov candidates for the class of PDEs that have form (\ref{pde:class}) using positive matrices.

\begin{mythm}\label{Th1}
Given any positive semi-definite matrix $P\in\s^{\frac{n}{2}(d+1)(d+4)}$ we can partition it as
\[
P=\left[\begin{array}{cc}
P_{11} & P_{12}  \\
P_{21} & P_{22}
\end{array}\right],\numberthis{\label{mat:P}}
\]
where $P_{11}\in\s^{n(d+1)}$.
Define
\[
Z_1(x):=Z_d(x)\otimes I_n\text{ and }Z_2(x,y):=Z_d(x,y)\otimes I_n\numberthis{\label{th1:Z1Z2}}
\]
where $x,y\in(a,b)$, $Z_d$ is a vector of monomials up to degree $d$ and $\otimes$ is the Kronecker product. If for some $\epsilon>0$
\begin{align*}
M(x):&=Z_1(x)^TP_{11}Z_1(x)+\epsilon I_n,\numberthis{\label{eq:M}}\\
N(x,y):&=Z_1(x)^T P_{12} Z_2(x,y)+Z_2(y,x)^TP_{21}Z_1(y)\\
&\hspace*{0.7in}+\int_a^b Z_2(z,x)^T P_{22}Z_2(z,y)\,dz,\numberthis{\label{eq:N}}
\end{align*}
then functional $V:L_2^n(a,b)\to\R$, defined as
\begin{align*}
V(w):&=\int_a^b w(x)^T M(x) w(x)\,dx\\
&\quad\quad+\int_a^b w(x)^T \int_a^b N(x,y) w(y)\,dydx,\numberthis{\label{eq:Lcand}}
\end{align*}
satisfies
\[
V(w)\geq\epsilon\|w\|_{L_2^n}\text{ for all }w\in L_2^n(a,b).\numberthis{\label{condV}}
\]
\end{mythm}

\begin{proof}
The idea of the proof is to show that $V$ from (\ref{eq:Lcand}), satisfies the following equation.
\begin{align*}
V(w)=&\int_a^b\hspace*{-0.05in} (\mathcal{Z}w)(x)^T P (\mathcal{Z}w)(x)\,dx + \epsilon\int_a^b \hspace*{-0.05in} w(x)^Tw(x)\,dx, \numberthis{\label{VZ}}
\end{align*}
where for all $x\in(a,b)$
\[
(\mathcal{Z}w)(x):=\left[\begin{array}{c}
Z_1(x)w(x)\\
\int_a^b Z_2(x,y)w(y)dy
\end{array}\right].\numberthis{\label{mathcZ}}
\]
Since $P\geq0$, then it is straightforward to show (\ref{condV}).

Consider the first integral of the right hand side in (\ref{VZ}), substitute for $\mathcal{Z}$ from (\ref{mathcZ}) and use the partition (\ref{mat:P}) as follows.
\begin{align*}
&\int_a^b (\mathcal{Z}w)(x)^T P (\mathcal{Z}w)(x)\,dx\numberthis{\label{prf:eq1}} \\
&\hspace*{0.0in}=\int_a^b w(x)^T Z_1(x)^T P_{11} Z_1(x) w(x)\,dx\\
&\hspace*{0.1in}+\int_a^b w(x)^T Z_1(x)^T P_{12} \int_a^b Z_2(x,y)w(y)\,dydx\\
&\hspace*{0.1in}+\int_a^b \int_a^b w(y)^TZ_2(x,y)^T\,dy\, P_{21} Z_1(x) w(x) dx\\
&\hspace*{0.1in}+\int_a^b\hspace*{-0.05in} \int_a^b w(y)^TZ_2(x,y)^T\,dy\, P_{22}\hspace*{-0.05in} \int_a^b\hspace*{-0.05in} Z_2(x,z)w(z)\,dzdx.
\end{align*}
Changing the order of integration in the 3rd integral of the right hand side of (\ref{prf:eq1}) and then switching between the integration variables $x$ and $y$ results in
\begin{align*}
&\int_a^b\int_a^b w(y)^TZ_2(x,y)^T\,dy\, P_{21} Z_1(x) w(x) dx\\
& = \int_a^b w(x)^T \int_a^b Z_2(y,x)^T P_{21} Z_1(y)w(y)\,dy dx.\numberthis{\label{prf:INT3}}
\end{align*}
Changing two times the order of integration in the 4th integral of the right hand side of (\ref{prf:eq1}) and then switching first between the integration variables $x$ and $z$, and then between $x$ and $y$ results in
\begin{align*}
&\hspace*{-0.05in}\int_a^b \hspace*{-0.05in}\int_a^b w(y)^TZ_2(x,y)^T\,dy\, P_{22}\hspace*{-0.05in} \int_a^b\hspace*{-0.05in} Z_2(x,z)w(z)\,dz\, dx\numberthis{\label{prf:INT4}}\\
&=\int_a^b\int_a^b\int_a^b  w(y)^T Z_2(x,y)^T P_{22} Z_2(x,z)w(z)dxdzdy\\
&=\int_a^b\int_a^b\int_a^b  w(y)^T Z_2(z,y)^T P_{22} Z_2(z,x)w(x)dzdxdy\\
&=\int_a^b\int_a^b\int_a^b  w(x)^T Z_2(z,x)^T P_{22} Z_2(z,y)w(y)dzdydx\\
&=\int_a^b w(x)^T\hspace*{-0.02in} \int_a^b\int_a^b  Z_2(z,x)^T  P _{22} Z_2(z,y)dzw(y)dydx.
\end{align*}
Using (\ref{prf:eq1})-(\ref{prf:INT4}) we can write
\begin{align*}
&\hspace*{-0.1in}\int_a^b (\mathcal{Z}w)(x)^T P (\mathcal{Z}w)(x)\,dx \\
&=\int_a^b w(x)^T Z_1(x)^T P_{11}Z_1(x)w(x)\,dx\\
&\hspace*{0.2in}+\int_a^b w(x)^T \int_a^b \Bigg( Z_1(x) P_{12}Z_2(x,y)\\
&\hspace{0.2in}+Z_2(y,x) P_{21}Z_1(y)\\
&\hspace*{0.2in}+\int_a^b Z_2(z,x)^T P_{22} Z_2(z,y)dz\Bigg)w(y)\,dy\,dx.\numberthis{\label{prf:gQ}}
\end{align*}
Using (\ref{eq:M}), (\ref{eq:N}) and (\ref{prf:gQ}) we can see that
\begin{align*}
&\hspace*{-0.05in}\int_a^b\hspace*{-0.05in} (\mathcal{Z}w)(x)^T\hspace*{-0.02in} P (\mathcal{Z}w)(x)\,dx\\
&\hspace*{0.3in}=\int_a^b w(x)^T M(x) w(x)\,dx - \epsilon\int_a^b w(x)^Tw(x)\,dx \\
&\hspace*{0.8in}+\int_a^b w(x)^T\int_a^b N(x,y) w(y)\,dydx.\numberthis{\label{prf:last}}
\end{align*}
If we add $\epsilon\int_a^b w(x)^Tw(x)\,dx$ to the both sides of (\ref{prf:last}) and use (\ref{eq:Lcand}), then we get (\ref{VZ}), which concludes the proof.
\end{proof}

\subsection{Functionals that are positive on $L_2^n(a,b)$ and not necessarily on $L_2^n(\R)$.}

Adding an extra term in (\ref{VZ}) as follows allow us to parameterize a larger set of Lyapunov candidates.
\begin{align*}
V(w)=&\int_a^b (\mathcal{Z}w)(x)^T P (\mathcal{Z}w)(x)\,dx + \epsilon\int_a^b w(x)^Tw(x)\,dx\\
&+\int_a^b g(x)(\mathcal{Z}w)(x)^T Q (\mathcal{Z}w)(x)\,dx, \numberthis{\label{VZ2}}
\end{align*}
where $g:[a,b]\to\R$ is continuous and positive and $Q\geq0$. Specifically, in this paper, we choose
\[
g(x)=(x-a)(b-x).\numberthis{\label{eq:g}}
\]

\begin{mythm}\label{Th2}
Given any positive semi-definite matrices $P,Q\in\s^{\frac{n}{2}(d+1)(d+4)}$ we can partition them as
\[
P=\left[\begin{array}{cc}
P_{11} & P_{12}  \\
P_{21} & P_{22}
\end{array}\right]\text{ and }Q=\left[\begin{array}{cc}
Q_{11} & Q_{12}  \\
Q_{21} & Q_{22}
\end{array}\right],\numberthis{\label{mat:PQ}}
\]
such that $P_{11},Q_{11} \in\s^{n(d+1)}$.
If for some $\epsilon>0$
\begin{align*}
\hspace*{-0.1in}M(x):&=Z_1(x)^T\left(P_{11}+g(x)Q_{11}\right)Z_1(x)+\epsilon I_n,\numberthis{\label{eq:M2}}\\
\hspace*{-0.1in}N(x,y):&=Z_1(x)^T\left(P_{12}+g(x)Q_{12}\right)Z_2(x,y)\numberthis{\label{eq:N2}}\\
&\quad+Z_2(y,x)^T\left(P_{21}+g(y)Q_{21}\right)Z_1(y)\\
&\quad+\hspace*{-0.025in}\int_a^b\hspace*{-0.025in} Z_2(z,x)^T\hspace*{-0.025in}\left(P_{22}+g(z)Q_{22}\right)Z_2(z,y)\,dz,
\end{align*}
where $Z_1,Z_2$ are defined as previously in (\ref{th1:Z1Z2}) and $g$ in (\ref{eq:g}), then functional $V:L_2^n(a,b)\to\R$, defined as in (\ref{eq:Lcand}) satisfies $V(w)\geq\epsilon\|w\|_{L_2^n}$ for all $w\in L_2^n(a,b)$.
\end{mythm}
\begin{proof}
The proof for Theorem \ref{Th2} is pretty much the same as for Theorem \ref{Th1} with (\ref{VZ2}) instead of (\ref{VZ}) and thus omitted here.
\end{proof}

For simplicity we define a set of polynomials $(M,N)$ as follows.
\[
\hspace*{-0.1in}\Sigma_+^{n,d,\epsilon}:=\{(M,N):\exists P,Q\geq0\text{ and }(\ref{eq:M2}),(\ref{eq:N2})\text{ hold}\}.\numberthis{\label{Sigma+}}
\]

Similarly, we can define another set of polynomials for some $\epsilon<0$ that parameterize functionals of the form (\ref{eq:Lcand}) such that $V(w)\leq\epsilon\|w\|_{L_2^n}$ for all $w\in L_2^n(a,b)$.
\[
\hspace*{-0.1in}\Sigma_-^{n,d,\epsilon}:=\{(M,N): (-M,-N)\in\Sigma_+^{n,d,-\epsilon}\}.\numberthis{\label{Sigma-}}
\]


\section{DERIVATIVE OF THE LYAPUNOV FUNCTIONAL}

For convenience we define the following functions. For all $x,y\in[a,b]$ and $t>0$
\begin{align*}
K(x):&=\left[\begin{array}{ccc}
K_{11}(x)&M(x)B(x)&M(x)A(x)\\
B(x)^TM(x)&0&0\\
A(x)^TM(x)&0&0
\end{array}\right],\\
L(x,y):&=\left[\begin{array}{ccc}
\hspace*{-0.1in} L_{11}(x,y) & \hspace*{-0.2in}N(x,y)B(y) & \hspace*{-0.1in}N(x,y)A(y)\hspace*{-0.02in}\\
B(x)^TN(x,y)&0&0\\
A(x)^TN(x,y)&0&0
\end{array}\right],\\
K_{11}(x)&=C(x)^TM(x)+M(x)C(x),\\
L_{11}(x,y)&=C(x)^TN(x,y)+N(x,y)C(y),\\
U(t,x):&=[\ u(t,x)^T\ u_x(t,x)^T\ u_{xx}(t,x)^T\ ]^T.\numberthis{\label{KLU}}
\end{align*}
If we substitute $u(t,\cdot)$ for $w$ in (\ref{eq:Lcand}), differentiate the result with respect to $t$ and do some straightforward manipulations, we can obtain
\begin{align*}
\frac{d}{dt}\left[ V(u(t,x)) \right]=&\int_a^b U(t,x)K(x)U(t,x)\,dx\\
&+\int_a^b U(t,x)L(x,y)U(t,y)\,dydx.\numberthis{\label{derV}}
\end{align*}

It is possible to check if $(K,L)\in\Sigma_-^{3n,d,0}$, but would be conservative. The reason is that the elements in $U$ are not independent, i.e. second and third elements are partial derivatives of the first one. Therefore it is enough to ask (\ref{derV}) to be negative only on a subspace of $L_2^{3n}(a,b)$, i.e.
\[
\Lambda\hspace*{-0.02in}=\hspace*{-0.03in}\left\{\hspace*{-0.05in}\left[\begin{array}{c}
w_1\\
w_2\\
w_3
\end{array}\right]\hspace*{-0.03in}\in\hspace*{-0.02in} L_2^{3n}(a,b):  D\hspace*{-0.03in}\left[\begin{array}{c}
w_1(a)\\
w_1(b)\\
w_2(a)\\
w_2(b)
\end{array}\right]\hspace*{-0.05in}=0,\hspace*{-0.05in}\begin{array}{c}
w_2=w'_1,\\
w_3=w''_1
\end{array}
\hspace*{-0.05in}
\right\}\numberthis{\label{Lambda}}
\]
Notice, that $\Lambda$ depends on $D$ that represents the boundary conditions as before in (\ref{eq:BC}).

\section{SPACING OPERATORS}

To parameterize functions which are negative on $\Lambda$, but not necessarily on the whole space $L_2^{3n}(a,b)$ we use the following theorem.


\begin{mythm}
Let $X$ be a closed subspace of some Hilbert space $Y$. Then $\left<u,\mathcal{R}u\right>_Y\leq 0$ for all $u\in X$ if and only if there exist $\mathcal{M}$ and $\mathcal{T}$ such that $\mathcal{R}=\mathcal{M}+\mathcal{T}$ and $\left<w,\mathcal{M}w\right>_Y\leq 0$ for all $w\in Y$ and $\left<u,\mathcal{T}u\right>_Y=0$ for all $u\in X$.
\end{mythm}
\begin{proof}
($\Leftarrow$) is straightforward. For ($\Rightarrow$), suppose $\left<u,\mathcal{R}u\right>_Y\leq0$ for all $u\in X$. Since $X$ is a closed subspace of a Hilbert space $Y$, there exists a projection operator such that $\mathcal{P}=\mathcal{P}^*=\mathcal{P}\mathcal{P}$ and $\mathcal{P}w\in X$ for all $w\in Y$. Let $\mathcal{M}=\mathcal{PRP}$ and $\mathcal{T=M-R}$. Then for all $w\in Y$,
\[
\left<w,\mathcal{M}w\right>_Y=\left<w,\mathcal{PRP}w\right>_Y=\left<\mathcal{P}w,\mathcal{RP}w\right>_Y\leq0
\]
since $\mathcal{P}w\in X$. Furthermore, for all $u\in X$
\begin{align*}
\left<u,\mathcal{T}u\right>_Y &= \left<u,\mathcal{PRP}u \right>_Y - \left<u,\mathcal{R}u \right>_Y \\
&= \left< \mathcal{P}u,\mathcal{RP}u \right>_Y - \left< u,\mathcal{R}u \right>_Y \\
&= \left< u,\mathcal{R}u \right>_Y - \left< u,\mathcal{R}u \right>_Y = 0.
\end{align*}
\end{proof}
We use $\Sigma_-^{3n,d,0}$ to parameterize a subset of $\mathcal{M}$ that is negative on the whole space $L_2^{3n}(a,b)$. Now we show how to parameterize a subset of operators $\mathcal{T}$ - the so-called ``spacing operators'' using polynomial spacing functions. Therefore the sum of $\mathcal{M}$ and $\mathcal{T}$ yield an operator $\mathcal{R}$ which is negative on $\Lambda$, but not necessarily on the whole $L_2^{3n}(a,b)$ space.

\subsection{Parametrization of Spacing Operators by Polynomials}

The following lemmas define the structure of polynomial matrices $T$ and $R$ such that for all $W\in\Lambda$
\begin{align*}
&\int_a^b W(x)^T T(x)W(x)dx \\
&+\int_a^b W(x)^T\int_a^b R(x,y)W(y)dydx=0.
\end{align*}
Beforehand define the following vector for convenience.
\[
\Upsilon:=[\ w(a)^T\ w(b)^T\ w'(a)^T\ w'(b)^T\ ]^T\numberthis{\label{upsilon}}
\]


\begin{mylem}\label{lem:sp1}
Let $P_i:[a,b]\to\R^{n\times n}, i=1..4$ be polynomials and $w,w',w''\in L_2^n(a,b)$. If
\[
T(x)\hspace*{-0.02in}=\hspace*{-0.02in}\left[\hspace*{-0.08in}\begin{array}{ccc}
P_1'(x)			& \hspace*{-0.02in} P_1(x)+P_2'(x) & \hspace*{-0.02in} P_2(x)\\
 P_1(x)+P_3'(x)	& \hspace*{-0.02in} P_2(x)+P_3(x)+P_4'(x) & \hspace*{-0.02in} P_4(x)\\
P_3(x)			& \hspace*{-0.02in} P_4(x) & \hspace*{-0.02in} 0
\end{array}
\hspace*{-0.05in}\right] \hspace*{-0.1in} \numberthis{\label{spT}}
\]
then
\[
\int_a^b W(x) T(x) W(x)\,dx = \Upsilon^T \Pi_1 \Upsilon,
\]
where $\Upsilon$ is defined in (\ref{upsilon}) and
\[
\Pi_1=\left[\begin{array}{cccc}
-P_1(a) & 0 & -P_2(a) & 0\\
0 & P_1(b) & 0 & P_2(a)\\
-P_3(a) & 0 & -P_4(a) & 0\\
0 & P_3(b) & 0 & P_4(b)
\end{array}\right].
\]
\end{mylem}
\vspace*{0.2in}
\begin{proof}
The proof is based on applying the fundamental theorem of calculus to
\[
\int_a^b\frac{d}{dx} \left(\left[\begin{array}{c}
w(x)^T\\
w'(x)^T
\end{array}\right]^T \left[\begin{array}{cc}
P_1(x) & P_2(x)\\
P_3(x) & P_4(x)
\end{array}\right]\left[\begin{array}{c}
w(x)\\
w'(x)
\end{array}\right]\right)dx.
\]
and omitted for brevity.
\end{proof}
Notice that $D\Upsilon=0$ and, therefore,
\begin{align*}
\Upsilon^T \Pi_1 \Upsilon
&= \Upsilon^T (I_{4n}-D+D)^T \Pi_1 (I_{4n}-D+D) \Upsilon\\
&= \Upsilon^T (I_{4n}-D+D)^T \Pi_1 (I_{4n}-D) \Upsilon\\
&= \Upsilon^T (I_{4n}-D)^T \Pi_1 (I_{4n}-D) \Upsilon.
\end{align*}
Using Lemma (\ref{lem:sp1}) we can define the following set.
\[
\Xi_1:=\{ T\text{ as defined in (\ref{spT})}: (I_{4n}-D)^T \Pi_1 (I_{4n}-D)=0\}
\]
Thus, for any $T\in\Xi_1$ and any $W\in\Lambda$ we have
\[
\int_a^b W(x)^T T(x)W(x)dx=0.
\]


\begin{mylem}\label{lem:sp2}
Let $Q_i:[a,b]\times[a,b]\to\R^{n\times n}, i=1..4$ be polynomials and $w,w',w''\in L_2^n(a,b)$.  If
\begin{align*}
R_1(x,y)&=\left[\begin{array}{ccc}
Q_{1,xy}(x,y) & R_{1,12}(x,y) & Q_{3,x}(x,y) \\
R_{1,21}(x,y) &	R_{1,22}(x,y) & R_{1,23}(x,y)\\
Q_{2,y}(x,y)  &	R_{1,32}(x,y) &	Q_4(x,y)
\end{array}\right],\\
R_{1,12}(x,y)&= Q_{3,xy}(x,y) + Q_{1,x}(x,y),\\
R_{1,21}(x,y)&= Q_{2,xy}(x,y) + Q_{1,y}(x,y),\\
R_{1,22}(x,y)&= Q_{4,xy}(x,y) + Q_{2,x}(x,y) + Q_{3,y}(x,y),\\
R_{1,23}(x,y)&= Q_{4,x}(x,y) + Q_3(x,y),\\
R_{1,32}(x,y)&= Q_{4,y}(x,y) + Q_2(x,y)\numberthis{\label{spR1}},
\end{align*}
then
\begin{align*}
&\int_a^b\int_a^b  W(x)^T R_1(x,y) W(y) dxdy = \Upsilon^T \Theta_1 \Upsilon,
\end{align*}
where $\Upsilon$ is defined in (\ref{upsilon}) and
\[
\Theta_1\hspace*{-0.03in}=\hspace*{-0.03in}\left[\begin{array}{cccc}
Q_1(a,a) & -Q_1(a,b) & Q_3(a,a) & -Q_3(a,b)\\
-Q_1(b,a) & Q_1(b,b) & -Q_3(b,a) & Q_3(b,b)\\
Q_2(a,a) & -Q_2(a,b) & Q_4(a,a) & -Q_4(a,b)\\
-Q_2(b,a) & Q_2(b,b) & -Q_4(b,a) & Q_4(b,b)
\end{array}\right]
\]
\end{mylem}
\vspace*{0.1in}
\begin{proof}
The proof is straightforward double application of the fundamental theorem of calculus to
\[
\int_a^b\hspace*{-0.07in}\int_a^b\hspace*{-0.07in}\frac{\partial^2}{\partial x\partial y} \hspace*{-0.03in}\left(\hspace*{-0.03in}\left[\hspace*{-0.05in}\begin{array}{c}
w(x)^T\\
w'(x)^T
\end{array}\hspace*{-0.07in}\right]^T \hspace*{-0.05in}\left[\hspace*{-0.05in}\begin{array}{cc}
Q_1(x,y) & \hspace*{-0.1in}Q_3(x,y)\\
Q_2(x,y) & \hspace*{-0.1in}Q_4(x,y)
\end{array}\hspace*{-0.07in}\right]\hspace*{-0.05in}\left[\hspace*{-0.07in}\begin{array}{c}
w(y)\\
w'(y)
\end{array}\hspace*{-0.07in}\right]\hspace*{-0.03in}\right)\hspace*{-0.03in}dxdy.
\]
\end{proof}
Similarly as for $\Xi_1$, using Lemma (\ref{lem:sp2}) we can define
\[
\Xi_2:=\{ R_1\text{ as defined in (\ref{spR1})}: (I_{4n}-D)^T \Theta_1 (I_{4n}-D)\hspace*{-0.02in}=\hspace*{-0.02in}0\}
\]
Thus, for any $R_1\in\Xi_2$ and any $W\in\Lambda$ we have
\[
\int_a^b \int_a^b W(x)^T R_1(x,y)W(y)dxdy=0.
\]


\begin{mylem}\label{lem:sp3}
Let $Q_5,Q_6:[a,b]\times[a,b]\to\R^{n\times n}$ be polynomials and $w,w',w''\in L_2^n(a,b)$.  If
\begin{align*}
R_2(x,y)&\hspace*{-0.03in}=\hspace*{-0.03in}\left[\hspace*{-0.05in}\begin{array}{ccc}
0 &\hspace*{-0.05in} 0 &\hspace*{-0.05in} 0 \\
0 & \hspace*{-0.05in}0 & \hspace*{-0.05in}0 \\
Q_{5,y}(x,y)	&\hspace*{-0.05in}	Q_{6,y}(x,y)+Q_5(x,y)	&	\hspace*{-0.05in}Q_6(x,y)
\end{array}\hspace*{-0.05in}\right]\numberthis{\label{spR2}}
\end{align*}
then
\begin{align*}
&\int_a^b\hspace*{-0.02in}\int_a^b\hspace*{-0.02in}  W(x)^T R_2(x,y) W(y) dxdy \hspace*{-0.02in}=\hspace*{-0.02in} \int_a^b \hspace*{-0.02in} w''(x)^T \Theta_2(x) \Upsilon dx,
\end{align*}
where $\Upsilon$ is defined in (\ref{upsilon}) and
\[
\Theta_2(x)\hspace*{-0.03in}=\left[\begin{array}{c}
-Q_5(x,a)\ \ Q_5(x,b) \ -Q_6(x,a)\ \ Q_6(x,b)
\end{array}\right].
\]
\end{mylem}
\begin{proof}
The fundamental theorem of calculus should be applied to
\[
\int_a^b\hspace*{-0.05in}\int_a^b\hspace*{-0.05in}\frac{\partial}{\partial y} \hspace*{-0.02in}\left(\hspace*{-0.02in}w''(x)^T \hspace*{-0.02in}\left[\hspace*{-0.02in}\begin{array}{cc}
Q_5(x,y) &\hspace*{-0.02in} Q_6(x,y)
\end{array}\hspace*{-0.02in}\right]\hspace*{-0.02in}\left[\begin{array}{c}
\hspace*{-0.01in}w(y)\\
\hspace*{-0.01in}w'(y)
\end{array}\hspace*{-0.01in}\right]\right)\hspace*{-0.02in}dxdy.
\]
\end{proof}
Using Lemma (\ref{lem:sp3}) we can define the following set.
\[
\Xi_3:=\left\{ R_2\text{ as defined in (\ref{spR2})}: \begin{array}{c}
\Theta_2(x)^T (I_{4n}-D)\hspace*{-0.02in}=\hspace*{-0.02in}0\\
\text{ for all }x\in(a,b)
\end{array}\right\}
\]
Thus, for any $R_1\in\Xi_1$ and any $W\in\Lambda$ we have
\[
\int_a^b \int_a^b W(x)^T R_2(x,y)W(y)dxdy=0.
\]


\begin{mylem}\label{lem:sp4}
Let $Q_7,Q_8:[a,b]\times[a,b]\to\R^{n\times n}$ be polynomials and $w,w',w''\in L_2^n(a,b)$.  If
\begin{align*}
R_3(x,y)&\hspace*{-0.03in}=\hspace*{-0.03in}\left[\begin{array}{ccc}
0 & 0 & Q_{7,x}(x,y) \\
0 & 0 & Q_{8,x}(x,y)+Q_7(x,y) \\
0 &	0 &	Q_8(x,y)
\end{array}\right]\numberthis{\label{spR3}}
\end{align*}
then
\begin{align*}
&\int_a^b\int_a^b  W(x)^T R_3(x,y) W(x) dx = \int_a^b \Upsilon^T \Theta_3(y) w''(y) dy,
\end{align*}
where $\Upsilon$ is defined in (\ref{upsilon}) and
\[
\Theta_3(y)\hspace*{-0.03in}=\left[\begin{array}{cccc}
-Q_7(a,y) \ \ Q_7(b,y) \ -Q_8(a,y) \ \ Q_8(b,y)
\end{array}\right].
\]
\end{mylem}
\begin{proof}
Apply the fundamental theorem of calculus to
\[
\int_a^b\hspace*{-0.05in}\int_a^b\hspace*{-0.05in}\frac{\partial}{\partial x} \hspace*{-0.02in}\left(\left[\begin{array}{c}
w(x)\\
w'(x)
\end{array}\right]^T \left[\begin{array}{c}
Q_7(x,y) \\
Q_8(x,y)
\end{array}\right]w''(y)\right)\hspace*{-0.02in}dxdy.
\]
\end{proof}
Using Lemma (\ref{lem:sp4}) we can define the following set.
\[
\Xi_4:=\left\{ R_3\text{ as defined in (\ref{spR3})}: \begin{array}{c}
(I_{4n}-D)^T\Theta_3(x) \hspace*{-0.02in}=\hspace*{-0.02in}0\\
\text{ for all }x\in(a,b)
\end{array}\right\}
\]
Thus, for any $R_3\in\Xi_1$ and any $W\in\Lambda$ we have
\[
\int_a^b \int_a^b W(x)^T R_3(x,y)W(y)dxdy=0.
\]
Finally, we can define a set of polynomials, similarly to (\ref{Sigma+}) and (\ref{Sigma-}).
\[
\hspace*{-0.1in}\Sigma_0^{n,d}:=\{(T,R):T\in \Xi_1\text{ and }R\in \sum_{i=2}^4 \Xi_i \}.\numberthis{\label{Sigma0}}
\]

\section{AN LMI CONDITION FOR STABILITY}

In this section we present feasibility problem, solution to which provides a Lyapunov function for (\ref{pde:class}).
\begin{mythm}
\label{thm:final}
Given System (\ref{pde:class}), if there exist
\begin{itemize}
\item $d\in\mathbb{N}$, $\epsilon_1>0$, $\epsilon_2<0$, $(M,N)\in\Sigma_+^{n,d,\epsilon_1}$,
\item $(T,R)\in\Sigma_0^{3n,2d+2+\gamma}$, $(H,G)\in\Sigma_-^{3n,d+\gamma,\epsilon_2}$
\end{itemize}
where $\gamma=\max\{\text{deg}(A),\text{deg}(B),\text{deg}(C)\}$such that for all $x,y\in(a,b)$
\begin{align*}
&\left[\begin{array}{ccc}
K_{11}(x)&M(x)B(x)&M(x)A(x)\\
B(x)^TM(x)&0&0\\
A(x)^TM(x)&0&0
\end{array}\right]\\
&\hspace*{2.2in}=T(x)+H(x),\\
&\left[\begin{array}{ccc}
\hspace*{-0.1in} L_{11}(x,y) & \hspace*{-0.2in}N(x,y)B(y) & \hspace*{-0.1in}N(x,y)A(y)\hspace*{-0.02in}\\
B(x)^TN(x,y)&0&0\\
A(x)^TN(x,y)&0&0
\end{array}\right]\\
&\hspace*{1.95in}=R(x,y)+G(x,y),\\
&K_{11}(x)=C(x)^TM(x)+M(x)C(x),\\
&L_{11}(x,y)=C(x)^TN(x,y)+N(x,y)C(y),
\end{align*}
then (\ref{pde:class}) is stable.
\end{mythm}
\begin{proof}
Suppose conditions of the Theorem \ref{thm:final} hold. Then $V$ as defined in (\ref{eq:Lcand}) satisfies (\ref{condV}). Since $M$ and $N$ are polynomials, they are continuous. Thus there exists $b\in\R$ such that
\[
V(w)\leq b\|w\|_{L_2^n}.
\]
According to (\ref{KLU}) and (\ref{derV}) the time derivative of $V$ satisfies
\[
\frac{d}{dt}[V(u(t,\cdot))]\leq \epsilon_2\|w\|_{L_2^n}
\]
and, therefore, we can apply Theorem \ref{thLyap}, which concludes the proof.
\end{proof}


\section{NUMERICAL RESULTS}
\subsection{Example 1: System of Decoupled PDEs.}
Consider the following parameterized coupled PDE. 
\[
u_t(t,x)=\left[\begin{array}{cc}
1 & 0 \\
0 & 1
\end{array}\right]u_{xx}(t,x)+\left[\begin{array}{cc}
\lambda & 0 \\
0 & \lambda
\end{array}\right]u(t,x).\numberthis{\label{ex1}}
\]
The boundary conditions are
\[
u(t,0)=\left[\begin{array}{c}
0\\0
\end{array}\right]\text{ and } u(t,1)=\left[\begin{array}{c}
0\\0
\end{array}\right].
\]
The numerical solution given by MATLAB PDEPE solver implies that for $\lambda=9.8$ (\ref{ex1}) is stable and for  $\lambda=9.9$ (\ref{ex1}) it is unstable. We applied the proposed algorithm to estimate  the maximum $\lambda$ for which (\ref{ex1}) is stable. The results are presented in Table \ref{table example 1}.

\begin{table}
\caption{Maximum $\lambda$ for which (\ref{ex1}) is stable based on the proposed algorithm for different degree $d$ with $\epsilon=0.001$.}
\label{table example 1}
\begin{center}
\begin{tabular}{|c||c|c|c|c|c|c|||c|}
\hline
$d$ & 1 & 2 & 3 & 4 & 5 & 6 &$\lambda_{num}$\\
\hline
$\lambda$ & 5 & 5.8 & 7.4 & 8.1 & 8.1 & 8.1 & 9.8 \\
\hline
\end{tabular}
\end{center}
\end{table}

\subsection{Example 2: System of Coupled PDEs.}
\[
u_t(t,x)=\left[\begin{array}{cc}
1 & 0 \\
0 & 1
\end{array}\right]u_{xx}(t,x)+\left[\begin{array}{cc}
\lambda & 1 \\
1 & \lambda
\end{array}\right]u(t,x)\numberthis{\label{ex2}}
\]
boundary conditions are
\[
u(t,0)=\left[\begin{array}{c}
0\\0
\end{array}\right]\text{ and } u(t,1)=\left[\begin{array}{c}
0\\0
\end{array}\right].
\]
The numerical solution given by MATLAB PDEPE solver yields that for $\lambda=8.8$ (\ref{ex2}) is stable and for  $\lambda=8.9$ (\ref{ex2}) is unstable. We applied the proposed algorithm to calculate the maximum $\lambda$ for which (\ref{ex2}) is stable. The results are presented in Table \ref{table example 2}.

\begin{table}
\caption{Maximum $\lambda$ for which (\ref{ex2}) is stable based on the proposed algorithm for different degree $d$ with $\epsilon=0.001$.}
\label{table example 2}
\begin{center}
\begin{tabular}{|c||c|c|c|c|c|c|||c|}
\hline
$d$ & 1 & 2 & 3 & 4 & 5 & 6 & $\lambda_{num}$\\
\hline
$\lambda$ & 4 & 5.8 & 6.9 & 7.2 & 7.4 & 7.4 & 8.8 \\
\hline
\end{tabular}
\end{center}
\end{table}

\subsection{Example 3: System of Coupled PDEs with Mixed Boundary Conditions.}
Now consider a third parameterized PDE.
\[
u_t(t,x)=\left[\begin{array}{cc}
1 & 0 \\
0 & 1
\end{array}\right]u_{xx}(t,x)+\left[\begin{array}{cc}
\lambda & \lambda \\
\lambda & \lambda
\end{array}\right]u(t,x)\numberthis{\label{ex3}}
\]
The boundary conditions are
\[
u_x(t,0)=\left[\begin{array}{c}
0\\0
\end{array}\right]\text{ and } u(t,1)=\left[\begin{array}{c}
0\\0
\end{array}\right].
\]
The numerical solution given by MATLAB PDEPE solver implies that for $\lambda=15.9$ (\ref{ex3}) is stable and for  $\lambda=16$ (\ref{ex3}) is unstable. We applied the proposed algorithm to calculate the maximum $\lambda$ for which (\ref{ex3}) is stable. The results are presented in Table \ref{table example 3}.

\begin{table}
\caption{Maximum $\lambda$ for which (\ref{ex3}) is stable based on the proposed algorithm for different degree $d$ with $\epsilon=0.001$.}
\label{table example 3}
\begin{center}
\begin{tabular}{|c||c|c|c|c|c|c|||c|}
\hline
$d$ & 1 & 2 & 3 & 4 & 5 & 6 & $\lambda_{num}$ \\
\hline
$\lambda$ & 8.6 & 12.7 & 13.9 & 14.4 & 14.6 & 14.7 & 15.9\\
\hline
\end{tabular}
\end{center}
\end{table}

\subsection{Example 4: System of Coupled PDEs with Spatially Dependent Coefficients.}

For our final example, we consider a coupled PDE with spatially varying coefficients. 
\begin{align*}
u_t(t,x)=&\left[\begin{array}{cc}
5x^2+4 & 0\\
2x^2+7x & 7x^2+6
\end{array}\right]u_{xx}(t,x) \\
&\hspace{0.35in}+\left[\begin{array}{cc}
1 & -4x\\
-3.5x^2 & 0
\end{array}\right]u_x(t,x) \\
&\hspace{0.9in}- \left[\begin{array}{cc}
x^2 & 3 \\
2x & 3x^2
\end{array}\right]u(t,x)
\end{align*}
for all $t>0$, $x\in(0,1)$. Also for all $t>0$,
\[
u(t,0)=\left[\begin{array}{c}
0\\0
\end{array}\right]\text{ and } u(t,1)=\left[\begin{array}{c}
0\\0
\end{array}\right].
\]
Although this PDE is not parameterized, it is stable and our algorithm verified this property using a Lyapunov function with polynomial degree $d = 4$.


\section{CONCLUSIONS AND FUTURE WORKS}

In this paper we have presented a computational framework for stability analysis of coupled linear PDEs with spatially varying coefficients. We parameterized positive SOS Lyapunov functionals defined by multiplier and integral operators which are positive on an interval of the real line. We have enforced negativity of the derivative using a combination of SOS and a parametrization of projection operators defined by the fundamental theorem of calculus. The result is an LMI test for stability which can be implemented using SOSTOOLS coupled with an SDP solver such as Mosek. We applied the proposed framework to several examples of systems of coupled linear PDEs with both constant and spatially varying coefficients and with both Dirichlet and Neumann boundary conditions. The numerical results agreed relatively well with results based on simulation. In future work, we will use this framework to study stability of models such as the accoustic wave equations as well as examine the problem of optimal control and estimation for systems of coupled PDEs. Another step includes extension of the framework to systems with multiple spatial variables as in \cite{meyer2015stability} and include semi-separable kernels to improve accuracy as in \cite{gahlawat2015convex}.

\bibliography{ifacconf}


\addtolength{\textheight}{-3cm}

\end{document}